\documentclass[11pt]{article}
\usepackage{amsthm}
\usepackage{array,amsmath,amssymb}

\usepackage{graphicx,subfigure}
\usepackage{comment}
\usepackage{xcolor}

\begin{document}

\newcommand{\di}{\displaystyle}
\def\qed{\ifhmode\unskip\nobreak\hfill$\Box$\bigskip\fi \ifmmode\eqno{Box}\fi}

\newtheorem{thm}{Theorem}
\newtheorem{lem}[thm]{Lemma}
\newtheorem{clm}[thm]{Claim}
\newtheorem{cor}[thm]{Corollary}
\newtheorem{conjecture}[thm]{Conjecture}
\newtheorem{fact}[thm]{Fact}
\newenvironment{thmbis}
  {\addtocounter{thm}{-1}%
   \renewcommand{\thethm}{\arabic{thm}$'$}%
   \begin{thm}}
  {\end{thm}}

\title{Graph Partitions Under Average Degree Constraint}

\author{Yan Wang\thanks{School of Mathematical Sciences, Shanghai Frontier Science Center of Modern Analysis (CMA-Shanghai), Shanghai Jiao Tong University, 200240, China,
{\tt yan.w@sjtu.edu.cn}}
\and
Hehui Wu\thanks{Shanghai Center for Mathematical Sciences, Fudan University,
Shanghai, 200438, China,
{\tt hhwu@fudan.edu.cn}}
~\thanks{Supported in part by National Natural Science Foundation of China grant 11931006, National Key Research and Development Program of China (Grant No. 2020YFA0713200), and the Shanghai Dawn Scholar Program grant 19SG01.}
~
}

\date{\today}

\maketitle

\begin{abstract}
In this paper, we prove that every graph with average degree at least $s+t+2$ has a vertex partition into two parts, such that one part has average degree at least $s$, and the other part has average degree at least $t$.
This solves a conjecture of Cs\'{o}ka, Lo, Norin, Wu and Yepremyan.
\end{abstract}

\section{Introduction}

Let $G$ be a graph.
For $A \subseteq V(G)$, we use $A$ to denote the induced subgraph $G[A]$ of $G$ when there is no confusion,
and $||A||$ (respectively, $|A|$) to denote the number of edges (respectively, vertices) in $G[A]$. 
For a vertex $v$ of $G$, we use $d(v)$ to denote its degree and $\delta(v)$ to denote the minimum degree of $G$.
For $A,B \subseteq V(G)$, let $E(A,B)$ denote the set of edges $ab$ in $G$ where $a \in A$ and $b \in B$.

Vertex coloring is one of the most important concepts in graph theory.
A graph $G$ is called $(d_1,\cdots,d_k)$-colorable if $V(G)$ can be partitioned into $V_1,V_2,\cdots,V_k$ where $G[V_i]$ has maximum degree at most $d_i$ for $i \in [k]$.
Such coloring is also called improper or defective in other literature.
In 1966, Lov\'{a}sz \cite{Lo66} showed that every graph $G$ is $(d_1,\cdots,d_k)$-colorable whenever $(d_1 +1)+\cdots+(d_k +1) \ge \Delta(G)+1$ and equality is attained by the complete graphs.
We refer the readers to \cite{BK77,BKY13,CCW86,CE19,CLO18,CSY19,Ge65,HKS09,KKX14,KKX16,Sh12} for improper coloring
and \cite{AE20,Ar87,BK14,BLM17,BLM20,CGJ97,FH94,Ra11,SN18} for defective coloring. 

A classic result of Stiebitz \cite{St96} showed that for non-negative integers $s$ and $t$, a graph $G$ with minimum degree $s+t+1$ can be partitioned into vertex disjoint subgraphs $G_1$ and $G_2$ such that $G_1$ has minimum degree at least $s$ and $G_2$ has minimum degree at least $t$.
This confirmed a conjecture of Thomassen \cite{Th83}.
The complete gragh $K_{s+t+2}$ also shows this result is tight.
Recently, other vertex partition problems with degree constraint in graphs have been studied extensively in \cite{BL16,BS02,Di00,GK04,HMYZ18,Ka98,LL20,LX17,LX22,MY19, SD02}, and \cite{Ba17, LX21, SS19, SX20, SX21} for those in multigraphs, and \cite{BTV07, GK00} for algorithmic aspect.
Inspired from the work of Stiebitz on partitions under minimum degree constraint and Lov\'{a}sz on partitions under maximum degree constraint, in this paper, we prove the analogous result on partition under average degree constraint, which was conjectured in \cite{CLNWY15} by Cs\'{o}ka, Lo, Norin, Yepremyan and the second author.

\begin{thm}
\label{main}
Let $s,t > 0$ be two reals. Let $G = (V,E)$ be a graph such that $||V|| \ge (s + t + 1)|V|$.
Then there exists a non-trivial partition $(A, B)$ of $V$ such that 
\begin{enumerate}
\item[(i)] $||A||-s|A| \ge 0$;
\item[(ii)] $||B||-t|B| \ge 0$.
\end{enumerate} 
\end{thm}

Since the average degree of $G$ equals $2||V||/|V|$, Theorem \ref{main} is equivalent to that every graph with average degree at least $s'+t'+2$ has a non-trivial vertex partition into two parts, such that one part has average degree at least $s'$, and the other part has average degree at least $t'$.

In \cite{CLNWY15}, the authors showed that if $||V|| > (s + t + 1)(|V| - 1)$, then there exists a non-null subsets $A,B$ of $V$ such that $||A||-s(|A| - 1) > 0$ and $||B||-t(|B|-1) > 0$.
They applied this result to prove a conjecture of Reed and Wood and a conjecture of Harvey and Wood on the extremal functions for disconnected minors. 
Note that Theorem \ref{main} strengthens several aspects of this result.

Note that the condition $||V|| \ge (s+t+1)|V|$ in Theorem \ref{main} is sharp in the sense that $s+t+1$ cannot be replaced by any smaller real number. 
To see this, we consider the following example.
Let $s,t$ be two positive integers, and $n$ be a sufficiently large integer. 
Let $H = (V,E)$ be a graph consisting of a clique of size $s +t + 1$ and $n - (s+t+1)$ isolated vertices and all the edges between them.
In other words, $V = X \dot\cup Y$ where $|X| = s + t + 1$; and $E = \{ uv | u \in X  \text{ or } v \in X \}$.
So 
$||V|| = \frac{(s+t+1)(s+t)}{2} + (n-s-t-1)(s+t+1)$.
One can check that 
$||V||/|V| = (s+t+1)(1-\frac{s+t+2}{2n})  \rightarrow s+t+1$ as $n \rightarrow \infty$.
Suppose there exists non-null vertex partition $V_1$ and $V_2$ of $V$ such that $||V_1|| \geq s |V_1|$, $||V_2|| \geq t |V_2|$.
Let $X_1 = V_1 \cap X$ and $X_2 = V_2 \cap X$.
So $|X_1| \le s$ or $|X_2| \le t$.
By symmetry, we may assume that $|X_1| \le s$.
Now $||V_1||/|V_1| = s-\frac{s|X_1|-(|X_1|(|X_1|-1))/2}{|V_1|} < s$, a contradiction.

We also comment that the classic result of Stiebtiz on minimum degree condition needs $s$ and $t$ to be integers. 
If $s$ and $t$ are fractional, the minimum degree condition becomes effectively $\lceil s \rceil + \lceil t \rceil + 1$, which is $\lceil s+t+2 \rceil$ in worse case.
In this sense, complete graph is also a sharp example for Theorem \ref{main}.

Theorem \ref{main} implies the following theorem, which is sometimes more convenient to use. 

\begin{thmbis}
\label{main_prime}
Let $s,t > 0$ be two reals. Let $G = (V,E)$ be a graph such that $||V|| > (s + t + 1)|V|$.
Then there exists a non-trivial partition $(A, B)$ of $V$ such that 
\begin{enumerate}
\item[(i)] $||A||-s|A| > 0$;
\item[(ii)] $||B||-t|B| > 0$.
\end{enumerate} 
\end{thmbis}

For a graph $H$, let $c(H)$ be the supremum of $||G||/|G|$ taken over all non-null graphs $G$ not containing $H$ as a minor. 
Note that Theorem \ref{main_prime} also implies the following corollary, which answers a question of Qian (see \cite{CLNWY15}).

\begin{cor}
\label{corollary_qian}
Let $H$ be a disjoint union of non-null graphs $H_1$ and $H_2$, then 
$$c(H) \leq c(H_1) + c(H_2) + 1.$$
\end{cor}

\begin{proof}
Let $G=(V,E)$ be a non-null graph such that $||V|| > (c(H_1)+c(H_2)+1)|V|$.
By Theorem \ref{main_prime}, there exist vertex disjoint non-null subgraphs $G_1$ and $G_2$ of $G$ such that $||V(G_1)|| > c(H_1)|V(G_1)|$ and $||V(G_2)|| > c(H_2)|V(G_2)|$.
By definition, $G_i$ contains $H_i$ as a minor for $i \in [2]$.
Then $G$ contains $H$ as a minor, which implies that $c(H) \le c(H_1) + c(H_2) + 1$.
\end{proof}

The proof of Theorem \ref{main} mostly adopts the structure of the proof in \cite{CLNWY15}: we consider the partition problem as an integer programming problem, and try to find an integer solution from its linear relaxation in two steps. In each step, we design an approximate objective function, and then use integer rounding method to get a partial solution. However, Theorem 7 in \cite{CLNWY15} can only guarantee to obtain two vertex disjoint parts, but not necessary a partition. To solve this problem, we introduce a strengthened theorem (Theorem \ref{main2}). Also, to obtain our improved result, we carefully redesigned two new objective functions, and finish the proof with more refined arguments.

\section{A strengthened version of Theorem \ref{main}}

First, we prove Theorem \ref{main} separately when $s$ or $t$ is no more than $1/2$.

\begin{clm}
\label{small_st}
Theorem \ref{main} is true when $s \le 1/2$ or $t \le 1/2$.
\end{clm}

\begin{proof}
By symmetry, suppose $t \le 1/2$.
Let $I$ be the set of isolated vertices in $G$. 
Let $x$ be a vertex of minimum degree in $V \backslash I$.
Note that $d(x) \leq 2||V|| / (|V| - |I|)$.
Let $y$ be a vertex adjacent to $x$. 
Let $A:= V \backslash \{x , y\}$ and $B := \{x,y\}$.
We have $||B|| = 1 \ge t|B|$ and 
\begin{equation*}
\begin{split}
||A|| 
&= ||V|| - 1 - (d(x) - 1) - (d(y) - 1) \\
&\geq ||V|| - 1 - (\frac{2||V||}{|V| - |I|} - 1) - (|V| - |I| - 1 - 1) \\
&= ||V|| - \frac{2||V||}{|V| - |I|} - (|V| - |I|) + 2.
\end{split}
\end{equation*}
Note that $||V|| \le \binom{|V| - |I|}{2} < \frac{{(|V| - |I|)}^2}{2}$,
so $|V| \geq |V| - |I| > \sqrt{2||V||}$. 
Thus, 
\begin{equation*}
\begin{split}
||A|| 
&\geq ||V|| - 2||V|| / |V| - |V| + 2 \\
&= (|V| - 2) (||V||/|V| - 1) \\
&\geq (|V| - 2)(s + t + 1 - 1) \\
&> s|A|.
\end{split}
\end{equation*}
Therefore, $(A, B)$ is a partition of $V$ as desired. 
\end{proof}

Let $X$ be a set of vertices and $s,t$ be two real numbers. 
We define $T(X)=\max\{0, ||X||-(s+t+1)|X|\}$.
Now we introduce a stronger version of Theorem \ref{main}.

\begin{thm}
\label{main2}
Let $s,t \in \mathbb{R}$ with $s,t > 1/2$.
If $G=(V,E)$ is a graph satisfying $||V||-(s+t+1)|V|\ge 0$, then there exist disjoint non-empty sets $A$ and $B$ such that 
\begin{enumerate}
\item[(1)] $||A||-s|A|\ge 0$;
\item[(2)] $||B||-t|B|\ge 0$;
\item[(3)] $||A||-s|A|+||B||-t|B|\ge T(A\cup B)-1$.
\end{enumerate} 
\end{thm}

It is not difficult to show that Theorem \ref{main2} implies Theorem \ref{main}.

\begin{proof}[Proof of Theorem \ref{main}]
By Claim \ref{small_st}, we may assume $s, t \ge 1/2$.
Let $A, B$ be sets given in the condition of Theorem \ref{main2}.
If $V = A \cup B$, then $(A,B)$ is the desired partition for Theorem \ref{main}.
Otherwise, let $C := V \backslash (A \cup B)$.
Since $||V|| \ge (s+t+1)|V|$ and $T(A \cup B) \ge ||A \cup B|| - (s +t + 1)|A \cup B|$, 
$$\begin{array} {ll} 
&(|E(A,C)| + ||C||) + (|E(B,C)| + ||C||)  \\
&= ||V|| + ||C|| - ||A \cup B||\\
&\ge (s + t + 1)(|A| + |B| + |C|) - (s +t +1)(|A| + |B|) - T(A \cup B) \\
&= (s +t + 1)|C| - T(A \cup B) \\
&\geq (s +t)|C| - (T(A \cup B) - 1) \\
&\ge (s +t)|C| - (||A|| - s|A| + ||B|| - t|B|) \text{\quad (by (3) of Theorem \ref{main2})} \\
& = (s(|A| + |C|) - ||A||) + (t(|B| + |C|) - ||B|| ).
\end{array}$$ 
So we have
$|E(A,C)| + ||C|| \ge s(|A| + |C|) - ||A||$
or 
$|E(B,C)| + ||C|| \ge t(|B| + |C|) - ||B||$.

Hence, by symmetry, we may assume $|E(A,C)| + ||C|| \ge s(|A| + |C|) - ||A||$.
Thus $||A \cup C|| = ||A|| + |E(A,C)| + ||C|| \ge s(|A| + |C|)$, which implies $(A \cup C, B)$ is a desired partition. 
\end{proof}

\section{Proof of Theorem \ref{main2}}

By contradiction, suppose Theorem \ref{main2} is false.
Let $G=(V,E)$ be a minimal counterexample.
So $G$ is non-empty and $||V||-(s+t+1)|V| \ge  0$.
By minimality of $G$, we have 
\begin{fact}
\label{fact5}
For any $\emptyset \subsetneq X \subsetneq G$, $||X||-(s+t+1)|X| < 0$.
\end{fact}

Let $p=\frac{s+1}{s+t+2}, \bar{p}=\frac{t+1}{s+t+2}$, $T=T(V)>0$ and $n=|V|$. 
Note that $|V| \ge 2s+2t+3$.

\begin{lem}
\label{LEM:BoundofT}
All of the following hold: 
\begin{enumerate}
    \item[(1)] $G$ has no clique of size at least $2s+2t+3$;
    \item[(2)] $\delta(G) > s+t+1+T(G)$. Thus, $|V| > s+t+2+T(G)$;
    \item[(3)] $0\le T(G) < s+t+2$.
\end{enumerate}
\end{lem}

\begin{proof}
First, we show (1). For, otherwise, let $C$ be a clique of size $\lceil 2s+2t+3 \rceil$ in $G$.
Now let $A$ and $B$ be disjoint cliques of size $\lceil 2s + 1\rceil$, $\lceil 2t +1 \rceil$ respectively (One can verify that $\lceil 2s + 1 \rceil + \lceil 2t + 1 \rceil \le \lceil 2s+2t+3 \rceil$).
Thus $||A||-s|A|=\frac{|A|}{2}(|A|-1-2s)\ge \max\{0, \frac{|A|}{2}((|A|-1-2s)+(|B|-1-2t)-1)\}$. Similarly, $||B||-t|B|\ge \max\{0, \frac{|B|}{2}((|A|-1-2s)+(|B|-1-2t)-1)\}$. Together, we have $||A||-s|A|+||B||-t|B|\ge \max\{0, \frac{|A|+|B|}{2}(|A|+|B|-1-2s-2t-2)\}= T(A\cup B)$. A contradiction.

Next, we show (2). For, otherwise, suppose there exists a vertex $v \in G$ with $d(v) \le s+t+1+T(G)$.
Then $|V \backslash \{v\}| \ge 1$ and $||V \backslash \{v\}|| = ||V|| - d(v) \ge (s+t+1)|V| + T(G) - (s+t+1+T(G)) = (s+t+1)|V \backslash \{v\}|$, which contradicts Fact \ref{fact5}.

Finally, we show (3). Since $(s+t+1)|V|+T(G)=||V|| \ge \delta(G)|V|/2 \ge \frac{(s+t+1+T(G))|V|}{2}$,  we have
$T(G) \le \frac{|V|}{|V|-2}(s+t+1).$ 
If $|V|>2s+2t+4$, then $T(G) < \frac{2s+2t+4}{2s+2t+2}(s+t+1)=s+t+2$; 
If $|V|\le 2s+2t+4$, then $T(G) < \delta(G)-s-t-1\le |V|-1-s-t-1 = s+t+2$.
\end{proof}

We consider the vertex partition problem as an integer programming problem. 
Let $V = \{v_1, v_2, . . . , v_n\}$, 
and let $\vec{x} := (x_{v_1},x_{v_2},...,x_{v_n})$  $\in [0,1]^n$.
We define the following functions: 
\begin{equation*}
\begin{split}
f_0(\vec{x}) &:=\sum_{uv \in E} x_u x_v - s\sum_{v \in V} x_v\\
g_0(\vec{x}) &:=\sum_{uv \in E} (1-x_u)(1-x_v)-t\sum_{v \in V} (1-x_v) 
\end{split}
\end{equation*}

\begin{clm}
\label{clm:goal}
Theorem \ref{main2} is true if there exists $\vec{x} \in \{0, 1\}^n$ such that 
\begin{itemize}
    \item[(i)] $\vec{x} \neq \vec{0}$ and $\vec{x} \neq \vec{1}$;
    \item[(ii)] $f_0(\vec{x})\geq 0$ and $g_0(\vec{x}) \ge 0$;
    \item[(iii)] $f_0(\vec{x}) + g_0(\vec{x}) \ge T - 1$.
\end{itemize}
\end{clm}

\begin{proof}
Let $A = \{v \in V |x_v = 1\}$ and $B = V \backslash A$. 
Since $\vec{x} \neq \vec{0}$ and $\vec{x} \neq \vec{1}$, we have $A \ne \emptyset$ and $B \ne \emptyset$.
Moreover, (ii) implies $||A|| \ge s|A|$ and $||B|| \ge t|B|$, and (iii) implies $(||A|| - s|A|)+(||B|| - t|B|) \ge T(A \cup B)-1$.
Then $(A,B)$ satisfies the conclusion of Theorem \ref{main2}.
\end{proof}

In the first step, we optimize the following functions $f_1$ and $g_1$ instead of $f_0$ and $g_0$:
\begin{equation*}
\begin{split}
f_1(\vec{x}) &:=\sum_{uv \in E} x_u x_v-(s+t+1)p\sum_{v \in V} x_v\\
g_1(\vec{x}) &:=\sum_{uv \in E} (1-x_u)(1-x_v)-(s+t+1)\bar{p}\sum_{v \in V} (1-x_v)
\end{split}
\end{equation*}

\begin{fact}\label{FAT:incr}
$f_1-g_1$ is an increasing function.
\end{fact}

\begin{proof}
We have
\begin{eqnarray*}
&&f_1(\vec{x})-g_1(\vec{x})\\
&=&\sum_{uv \in E}(x_u+x_v-1)-(s+t+1)\sum_{v\in V}(px_v-\bar{p}(1-x_v))\\
&=&\sum_{v\in V}(x_v(d(v)-(s+t+1)(p+\bar{p}))+(s+t+1)\bar{p})-e(G)\\
&=&\sum_{v\in V}(x_v(d(v)-(s+t+1))+(s+t+1)\bar{p})-e(G).
\end{eqnarray*}
As $\delta(G)>s+t+1+T$, $f_1-g_1$ is an increasing function.
\end{proof}

For $\vec{x} \in [0,1]^n$, let $fr(\vec{x})=\{v \in V: 0<x_v<1\}$ denote the set of vertices corresponding to the non-integral values of $\vec{x}$. 
For $\vec{x} \in [0,1]^n$, let $\overrightarrow{x_C}$ be its restriction on $fr(\vec{x})$, i.e. $(x_C)_v = x_v$ if $v \in fr(\vec{x})$ and otherwise $0$. 
We show that we can achieve a fractional solution with fractional part being a clique.

\begin{lem}
\label{lem2}
There exists $\vec{y} \in [0,1]^n$ such that
\begin{enumerate}
    \item[(1)] $f_1(\vec{y}) \geq p^2 T, g_1(\vec{y})\ge \bar{p}^2T, \vec{y}\not\in\{\vec 0,\vec1\}$;
    \item[(2)] $C := fr(\vec{y})$ is a clique in $G$;
    \item[(3)] $\sum_{v \in C} (y_v-y_v^2)< (2s+2t+3)p(1-p)+\frac{T}{s+t+2}p(1-p)$; 
    \item[(4)] $\sum_{v \in V} y_v \ge 2(s+t+1)p$ and  
    $\sum_{v \in V} (1-y_v) \ge 2(s+t+1)\bar{p}.$
\end{enumerate}
\end{lem}

\begin{proof}
Without loss of generality, we may assume $s \le t$.
Let $\vec{y}\in [0,1]^n$ be chosen so that the following conditions are satisfied in order
\begin{itemize}
\item[(i)] $f_1(\vec{y}) \geq p^2 T, g_1(\vec{y}) \geq \bar{p}^2 T, \vec{y}\not\in\{\vec 0,\vec1\}$;
\item[(ii)] $|fr(\vec{y})|$ is minimum;
\item[(iii)]  $\sum_{v \in V} y_v$ is minimum.
\end{itemize}

Such $\vec{y}$ exists because $\vec{y}=p\vec{1}$ satisfies condition (i):
$f_1(p\vec{1}) = p^2 ||G|| - (s+t+1)p^2 |G|= p^2  ((s+t+1) |G| + T - (s+t+1) |G|)= p^2 T$. Similarly $g_1(p\vec{1})  = \bar{p}^2 T$.

First note that (i) implies (4). By (i), we have $f_1(\vec{y}) \ge p^2 T \ge 0$. so 
$0\le \sum_{uv \in E}y_uy_v-(s+t+1)p\sum_{v\in V}y_v< \frac{1}{2} (\sum_{v \in V} y_v)^2- (s+t+1)p \sum_{v \in V} y_v$, which implies  $\sum_{v \in V} y_v \ge 2(s+t+1)p$. Similarly $\sum_{v \in V} (1-y_v) \ge 2(s+t+1)\bar{p}$. 

We claim that $C:= fr(\vec{y})$ is a clique. 
Suppose for a contradiction that there exist $u,v \in fr(\vec{y})$ such that $uv \not\in E$. 
Then $f_1(\vec{y})$ and $g_1(\vec{y})$ are linear functions of $(y_u, y_v)$.
Hence, there exist $\alpha > 0$ and $\vec{r} \in \mathbb{R}^n \backslash \{\vec{0}\}$ with $r_k = 0$ for $k \not\in \{u,v\}$ such that $f_1(\vec{y} + \alpha \vec{r}) \geq f_1(\vec{y})$, $g_1(\vec{y} + \alpha \vec{r}) \geq g_1(\vec{y})$ and $0 \leq y_k + \alpha r_k \leq 1$ for $k \in V$.
Furthermore, we can choose $\alpha$ so that $y_u + \alpha r_u \in \{0,1\}$ or $y_v + \alpha r_v \in \{0,1\}$, and denote $\vec{y'} = \vec{y} + \alpha \vec{r}$. We have $f_1(\vec{y'}) \ge pT^2$ and $g_1(\vec{y'})\ge \bar{p}T^2$. Note that as $\vec{y}$ satisfied condition (i), hence also satisfied (4), $\sum_{v\in V}y_v \ge 2(s+t+1)p =  2\frac{(s+1)(s+t+1)}{s+t+2}>2$ as $s, t> 1/2$. This implies  $\vec{y'}\not=\vec{0}$. Similarly $\vec{y'}\not=\vec{1}$. Therefore, $\vec{y'}$ satisfies condition (i), but $|fr(\vec{y'})| < |fr(\vec{y})|$, which contradicts our choice of $\vec{y}$.

Now we just need to show (3).
Let $f'_v(\vec{x}) := \frac{\partial f_1(\vec{x})}{\partial x_v}$ and $g'_v(\vec{x}):= \frac{\partial g_1(\vec{x})}{\partial x_v}$. Let $\vec{e_v}$ be indicative vector, i.e. $(e_v)_u = 1$ if $u=v$ and otherwise $0$. Let $\vec{y'} := \vec{y} + \alpha \vec{e_v}$ for some $\alpha\not=0$. We have 
$$f_1(\vec{y'}) = f_1(\vec{y}) + \alpha f'_v(\vec{y}),$$ 
$$g_1(\vec{y'}) = g_1(\vec{y}) + \alpha  g'_v(\vec{y})$$

For $v \in fr(\vec{y})$, we have $f'_v(\vec{y})g'_v(\vec{y})<0$. Otherwise we can choose a proper $\alpha\not=0$, such that $f_1(\vec{y'})\ge f_1(\vec{y})\ge p^2 T$ and $g_1(\vec{y'})\ge g_1(\vec{y})\ge \bar{p}^2 T$, and furthermore $y'\in [0,1]^n \backslash \{\vec{0},\vec{1}\}$ but $|fr(\vec{y'})|<|fr(\vec{y})|$, which contradicts  the choice of $\vec{y}$ with condition (ii).
Moreover, by Fact~\ref{FAT:incr}, $f'_v(\vec{y})- g'_v(\vec{y})\ge 0$, we have $f'_v(\vec{y})> 0$ and $g'_v(\vec{y})< 0$.

Furthermore,  we have $f_1(\vec{y}) = p^2 T$. 
For, otherwise, suppose $f_1(\vec{y}) > p^2 T$, 
let $v \in fr(\vec{y})$ and $\vec{y'} = \vec{y} + \alpha \vec{e_v}$.
We may choose  $\alpha<0$ with $|\alpha|$ sufficiently small such that $f_1(\vec{y'})>p^2 T$, and we have $g_1(\vec{y'})>g_1(\vec{y}) \ge p^2 T$, which contradicts the choice of $\vec{y}$ with condition (iii).

We also have $f_1(\vec{y}_C)\leq p^2 T$.
Otherwise, $\vec{y}_C<\vec{y}$ (as $f_1(\vec{y}_C) > f_1(\vec{y})$), and by Fact~\ref{FAT:incr}, $f_1(\vec{y}_C)-g_1(\vec{y}_C) \le f_1(\vec{y})-g_1(\vec{y})$, we have $g_1(\vec{y}_C)\ge g_1(\vec{y})+(f_1(\vec{y}_C)-f_1(\vec{y}))> \bar{p}^2T+(Tp^2-Tp^2)=\bar{p}^2T$. Therefore
$\vec{y}_C$ satisfies condition (i), which contradicts to the choice of $\vec{y}$ with condition (ii).

Let $N = 2s+2t+3$.
Since
\begin{eqnarray*}
2p^2T &\ge& 2f_1(\vec{y}_C) \\
&=& \sum_{u,v \in C} 2y_uy_v - 2(s+t+1)p \sum_{v \in C} y_v\\
&=&(\sum_{v \in C}y_v)^2-\sum_{v \in C}y_v^2-(N-1)p\sum_{v\in C}y_v\\
&=&\sum_{v\in C}(y_v-y_v^2)+(\sum_{v\in C}y_v)^2-((N-1)p+1)\sum_{v\in C}y_v
\end{eqnarray*}
we have $$\sum_{v\in C}(y_v-y_v^2)\le -(\sum_{v\in C}y_v)^2+((N-1)p+1)\sum_{v\in C}y_v+2p^2T.$$

On the other hand, as $C < N=2s+2t+3$, we have
$$\sum_{v \in C} (y_v-y_v^2)
\le (\sum_{v \in C} y_v)-\frac{(\sum_{v \in C} y_v)^2}{|C|}
< (\sum_{v \in C} y_v)-\frac{(\sum_{v \in C} y_v)^2}{N}.$$

Let $x=\sum_{v \in C}y_v$, together we obtain
$$\sum_{v\in C}(y_v-y_v^2)\le \min\{-x^2+((N-1)p+1)x+2p^2T,-\frac{x^2}{N}+x\}.$$

The two curves $L_1: y=-x^2+((N-1)p+1)x+2p^2T$ and $L_2: y=-\frac{x^2}{N}+x$ are both concave down, and they intersect at two points :
$$\begin{array}{ll}
&-x^2+((N-1)p+1)x+2p^2T=-\frac{x^2}{N}+x \\
\Leftrightarrow &(1-\frac 1N)x^2-(N-1)px-2p^2T=0\\
\Leftrightarrow &x^2-Npx-\frac{2p^2TN}{N-1}=0\\
\Leftrightarrow &(x-(N+\epsilon)p)(x+\epsilon p)=0.\\
\end{array}$$
where $\epsilon>0$ satisfies $(N+\epsilon)\epsilon=\frac{2TN}{N-1}$.

We may assume $s \le t$, so $p \le 1/2$.
Note that $L_2: y=-\frac{x^2}{N}+x$ is concave down with value $Np(1-p)$ and  derivative $1-2p$ at $x=Np$, we have $\sum_{v\in C}(y_v-y_v^2)\le Np(1-p)+(1-2p)p\epsilon $ if $x\le Np+\epsilon p$. However $L_1$ and $L_2$ intersect at $x=Np+\epsilon p$, and $L_1: y=-x^2+((N-1)p+1)x+2p^2T$ is decreasing when $x\ge Np+\epsilon p$. So we always have $\sum_{v\in C}(y_v-y_v^2)\le Np(1-p)+(1-2p)p\epsilon $.

Since $(N+\epsilon)\epsilon=\frac{2TN}{N-1}$, we have $\epsilon < \frac{2T}{N-1}$. Therefore
$$\epsilon p(1-2p) < \frac{2T}{N-1}p(1-p)\frac{1-2p}{1-p}=\frac{2T}{N+1}p(1-p)\frac{t-s}{t+1}\frac{s+t+2}{s+t+1} \le \frac{2T}{N+1}p(1-p).$$

From above, we have $\vec{y}$ satisfies (3):
$$\begin{array}{ll}
\sum_{v \in C} (y_v-y_v^2) 
&<  Np(1-p)+\frac{2T}{N+1}p(1-p) \\
&=  (2s+2t+3)p(1-p)+\frac{T}{s+t+2}p(1-p) \\
&< (2s+2t+4)p(1-p).
\end{array}$$
\end{proof}

From now on, we fix $\vec{y} \in [0,1]^n$ that satisfies the conclusion of Lemma \ref{lem2}. 
In particular, we have $C := fr(\vec{y})$ is a clique. 
We define
$$f_2(\vec{x}) :=f_0(\vec{x})-\frac{((\sum_C x_v)-(\sum_C y_v)+\bar{p})^2}{2}-\frac{\sum_C (x_v-x_v^2)}{2};$$
$$g_2(\vec{x}) :=g_0(\vec{x})-\frac{((\sum_C x_v)-(\sum_C y_v)+p))^2}{2}-\frac{\sum_C (x_v-x_v^2)}{2}.$$
Note that $f_2$ and $g_2$ are lower bounds for $f_0$ and $g_0$ respectively. 
As $C$ is a clique by Lemma \ref{lem2} (2), it is not difficult to see that $f_2(\vec{x})$ and $g_2(\vec{x})$ are linear functions of $\vec{x}|_C$. 
Similar to the proof of Lemma \ref{lem2}, we have the following

\begin{lem}
\label{lem3}
There exists $\vec{z} \in [0,1]^n$ such that
\begin{enumerate}
    \item[(1)] $\vec{z}|_{V - C} = \vec{y}|_{V-C}$ and $|fr(\vec{z})|\le 1$;
    \item[(2)] $f_2(\vec{z}) \ge f_2(\vec{y})$ and 
$g_2(\vec{z})\ge g_2(\vec{y}).$
\end{enumerate}
\end{lem}

Let us fix $\vec{z}$ that satisfies Lemma \ref{lem3}.
We will see that $\vec{z}$ is almost a solution except for at most one non-integral vertex. 
We define
$$\begin{array} {ll}
&X:= p^2 T -\frac{p(1-p)T}{2s+2t+4} \geq p^2 T - \frac{p\bar{p}}{2}; \\
&A:=\sum_{v \in V} y_v-(s+t+1)p-\frac{1}{2}; \\
&Y:= \bar{p}^2 T -\frac{p(1-p)T}{2s+2t+4} \ge \bar{p}^2 T - \frac{p\bar{p}}{2}; \\
&B:=\sum_{v \in V} (1-y_v)-(s+t+1)\bar{p}-\frac{1}{2}.
\end{array}$$

\begin{lem}
\label{lem4}
All the followings hold:
\begin{enumerate}
    \item[(1)] $X \ge 0$, $Y \ge 0$, $A \ge (s+t+1)p-\frac{1}{2}> 0$, $B \ge (s+t+1)\bar{p}-\frac{1}{2} > 0$ and $A + B = |V| - (s+t+2) \ge T$;
    \item[(2)] $f_2(\vec{z}) \ge X + A\bar{p} > 0$ and $g_2(\vec{z}) \ge Y + Bp > 0$;
    \item[(3)] $f_2(\vec{z}) + g_2(\vec{z}) \ge X+Y+A\bar{p}+Bp \ge T - \frac{7}{12}$.
\end{enumerate}
\end{lem}

\begin{proof}
First we show (1). 
$X =  p^2 T -\frac{p(1-p)T}{2s+2t+4} = pT ( \frac{s+1}{s+t+2} - \frac{1-p}{2s+2t+4}) \ge 0$.
Similarly, $Y \ge 0$.

By Lemma \ref{lem2} (4), $\sum_V y_v \ge 2(s+t+1)p$ and $\sum_V (1 - y_v) \ge 2(s+t+1)\bar{p}$.
So $A = \sum_V y_v -(s+t+1)p-\frac{1}{2} \ge (s+t+1)p - \frac12=s+\bar{p}-\frac12 > 0$. 
Similarly we have $B\ge (s+t+1)\bar{p}-\frac12 > 0$.

Moreover, $A+B=\sum_{V}y_v-((s+t+1)p+1/2)+\sum_{V}(1-y_v)-((s+t+1)\bar{p}+\frac12)=|V|-(s+t+2) > T$ by Lemma \ref{LEM:BoundofT} (2).

Next we show (2). 
By Lemma \ref{lem2} (1)(3)(4) and Lemma \ref{lem3} (2), 
$$\begin{array} {ll}
f_2(\vec{z})
&\ge f_2(\vec{y}) \\
&= f_0(\vec{y}) - \frac{\bar{p}^2}{2} - \frac{\sum_C (y_v-y_v^2)}{2}\\
&= f_1(\vec{y}) + \bar{p} \sum_V y_v - \frac{\bar{p}^2}{2} - \frac{\sum_C (y_v-y_v^2)}{2} \\
&\ge p^2 T + \bar{p} \sum_V y_v - \frac {1}{2}(1-p)\bar{p} - \frac{1}{2}(2s+2t+3)p\bar{p} - \frac{p(1-p)T}{2s+2t+4} \\
&= p^2 T - \frac{p(1-p)T}{2s+2t+4} + (\sum_V y_v - (s+t+1)p-\frac{1}{2})\bar{p} \\
&= X + A\bar{p}>0\\
\end{array}$$
Similarly, $g_2(\vec{z}) \geq Y + Bp > 0$.

Finally we show (3). 
We may assume $s \le t$, hence $p=\frac{s+1}{s+t+2} \le \frac12$. By (1), we have
$$\begin{array}{ll}
&X+Y+A\bar{p}+Bp\\
\ge& Tp^2-\frac{p\bar{p}}{2}+T\bar{p}^2-\frac{p\bar{p}}{2}+p(A+B)+(1-2p)A\\
\ge& Tp^2+T\bar{p}^2-p\bar{p}+pT+(1-2p)((s+t+1)p-\frac12) \text{\quad (as $A+B \ge T$)}\\
\ge& Tp^2+T\bar{p}^2+pT-p(1-p)+(1-2p)((T-1)p-\frac12) \text{\quad (as $T \le s+t+2$)}\\
=& Tp^2+T\bar{p}^2+pT+Tp(1-2p)-p+p^2-(1-2p)\frac{1+2p}{2}\\
=&Tp^2+T\bar{p}^2+2Tp(1-p)-\frac{1}{2}+3p^2-p\\
=&T+3(p-\frac{1}{6})^2-\frac{1}{12}-\frac{1}{2}\\
\ge&T-\frac{7}{12} 
\end{array}$$
\end{proof}

By Lemma~\ref{lem3}~(1), we may assume $z_v\in \{0,1\}$ when $v\not=w$ for some vertex $w$. 
Let $\hat{z}^+$, $\hat{z}^-$ be two vectors obtained from $\hat{z}$ by setting $\hat{z_w}^+ = 1$, $\hat{z_w}^- = 0$ and $\hat{z_v}^+ = \hat{z_v}^- = z_v$ for $v \in V - \{w\}$. 

We are going to show that either $\hat{z}^+$ or $\hat{z}^-$ satisfies the conditions in Claim~\ref{clm:goal}, hence complete the proof of our main theorem.

\begin{proof}[Proof of Theorem~\ref{main2}]
As $f_0\ge f_2$ and $g_0\ge g_2$ in $[0,1]^n$,
it suffices to show that at least one of the followings holds:
\begin{enumerate}
    \item[(1)] $f_2(\hat{z}^+) > 0$, $g_2(\hat{z}^+) > 0$ and $f_2(\hat{z}^+) + g_2(\hat{z}^+) > T-1$; 
    \item[(2)] $f_2(\hat{z}^-) > 0$, $g_2(\hat{z}^-) > 0$ and $f_2(\hat{z}^-) + g_2(\hat{z}^-) > T-1$.
\end{enumerate}

We need the following notations in the proof: 
$$\begin{array} {ll}
&A' := A - \sum_{v \in V - \{w\}, vw\not\in E}y_v, \\
&B' := B - \sum_{v \in V - \{w\}, vw\not\in E}(1-y_v). 
\end{array}$$
Since $\vec{y}|_{V-C} = \vec{z}|_{V-C}$, we have
$$\begin{array} {ll}
\frac{\partial f_2}{\partial z_w}(\vec{z}) 
&=\sum_{wv \in E} z_v - s - (\sum_{v \in V - \{w\}} z_v - \sum_{v \in V} y_v + \bar{p}) - \frac12 \\
&= \sum_{v \in V} y_v - (s + \bar{p})- \frac12 - \sum_{v \in V - \{w\}, vw \not\in E} z_v \\
&= \sum_{v \in V} y_v - (s+t+1)p - \frac{1}{2} - \sum_{v \in V - \{w\}, vw \not\in E} y_v \\
&= A'.
\end{array}$$
Similarly, $\frac{\partial g_2}{\partial z_w}(\vec{z}) = -B'$.
Hence, we have the following
$$\begin{array} {ll}
&f_2(\hat{z}^+) \ge X+A\bar{p}+A'(1-z_w),\\
&g_2(\hat{z}^+) \ge Y+Bp-B'(1-z_w),\\
&f_2(\hat{z}^-) \ge X+A\bar{p}-A'z_w,\\
&g_2(\hat{z}^-) \ge Y+Bp+B'z_w.
\end{array}$$
Note $A' + B' = |V| - (s+t+2) - \sum_{v \in V - \{w\}, vw \not\in E} 1 = d(w) + 1 - (s+t+2) \ge \delta(G) - (s+t+1) > T \ge 0$ by Lemma \ref{LEM:BoundofT} (2). 
So at least one of $A', B'$ is positive. 

We may assume $A', B' >  0.$  For, otherwise, suppose $A' > 0, B' \le 0$ by symmetry. 
Then by Lemma \ref{lem4} (2)(3), $f_2(\hat{z}^+)  \ge X + A\bar{p} > 0$,  $g_2(\hat{z}^+) \ge Y + Bp > 0$, and $f_2(\hat{z}^+) + g_2(\hat{z}^+) \ge X+Y+A\bar{p}+Bp > T - \frac{7}{12}$.
So $\hat{z}^+$ is as desired.

Without loss of generality, we may assume $A' \ge B' > 0$.

If $z_w > 1-\frac{Y+Bp}{B'}$, we have $f_2(\hat{z}^+)=f_2(\hat{z})+A'(1-z_w)\ge X+A\bar{p} > 0$,
$g_2(\hat{z}^+) = g_2(\hat{z})-B'(1-z_w) > Y+Bp-B'(\frac{Y+Bp}{B'})\ge 0$,
and $f_2(\hat{z}^+)+g_2(\hat{z}^+)=f_2(\hat{z})+g_2(\hat{z})+(A'-B')(1-z_w)\ge X+A\bar{p} + Y+Bp \ge T-\frac{7}{12}$.

So we may assume $0\le z_w \le 1-\frac{Y+Bp}{B'}$. This implies that $B'\ge Y+Bp\ge Y+B'p$, hence $B'\ge Y/{\bar{p}}$.

Similarly to above, we have $g_2(\hat{z}^-)=g_2(\hat{z})+B'z_w \ge Y + Bp \ge 0$, and as $z_w\le 1-\frac{Y+Bp}{B'}\le 1-p = \bar{p} \le \frac{X+A\bar{p}}{A'}$, we have 
$f_2(\hat{z}^-)=f_2(\hat{z})-A' z_w \ge X+A\bar{p}-A' \frac{X+A\bar{p}}{A'} \ge 0$.

It is just left to show that $f_2(\hat{z}^-)+g_2(\hat{z}^-)\ge T-1$.
Since $A \ge A'$ and $B \ge B'$, we have
$$\begin{array}{ll}
&f_2(\hat{z}^-)+g_2(\hat{z}^-)\\
\ge& X+Y+A\bar{p}+Bp-(A'-B')(1-\frac{Y+Bp}{B'})\\
=&X+A\bar{p}-A'(1-\frac{Y+Bp}{B'})+B'\\
\ge&X+A\bar{p}-A(1-\frac{Y+Bp}{B'})+B'\\
\ge&X+A\bar{p}-A(1-\frac{Y+B'p}{B'})+B'\\
=&X+\frac{AY}{B'}+B'\\
\end{array}$$

If $AY\ge (Y/\bar{p})^2$, then 

$$\begin{array}{lll}
&&X+\frac{AY}{B'}+B'\\
&\ge& X+2\sqrt{AY}\\
&\ge& X+2Y/\bar{p}\\
&\ge& Tp^2-\frac{p\bar{p}}2+2(T\bar{p}^2-\frac{p\bar{p}}{2})/\bar{p}\\
&=&T(p^2+2-2p)-(\frac{p\bar{p}}{2}+p)\\
&=&T(1+(1-p)^2)-\frac{p(3-p)}{2}\\
&\ge& T-1.
\end{array}$$

If $AY\le (Y/\bar{p})^2$, then since $B' \ge Y/\bar{p} \ge \sqrt{AY}$, we have

$$\begin{array}{lll}
&&X+\frac{AY}{B'}+B'\\
&\ge& X+\frac{AY}{Y/\bar{p}}+Y/\bar{p}\\
&\ge& X+A\bar{p}+Y/\bar{p}\\
&\ge& Tp^2-\frac{p\bar{p}}2+((s+t+1)p-\frac12)\bar{p}+(T\bar{p}^2-\frac{p\bar{p}}{2})/\bar{p}\\
&>& Tp^2-\frac{p\bar{p}}2+((T-1)p-\frac12)\bar{p}+(T\bar{p}^2-\frac{p\bar{p}}{2})/\bar{p}\\
&=&T(p^2+p\bar{p}+\bar{p})-(\frac{3p\bar{p}}{2}+\frac{\bar{p}}{2}+\frac{p}2)\\
&=&T(p+\bar{p})-(\frac{3p\bar{p}}{2}+\frac12)\\
&\ge& T-\frac 78.
\end{array}$$
\end{proof}


\begin{thebibliography}{99}

\bibitem{AE20}
A. Armstrong and N. Eaton. New restrictions on defective coloring with applications to Steinberg-type graphs. \textit{J. Comb. Optim.} \textbf{40} (2020), no. 1, 181--204.

\bibitem{Ar87}
D. Archdeacon. A note on defective colorings of graphs in surfaces. \textit{J. Graph Theory}, \textbf{11} (1987), no. 4, 517--519.

\bibitem{Ba17}
A. Ban. Decomposing weighted graphs. \textit{J. Graph Theory}, \textbf{86} (2017),  250--254.

\bibitem{BK77}
O.V. Borodin and A.V. Kostochka. On an upper bound of a graph’s chromatic number, depending on the graph’s degree and density. \textit{J. Comb. Theory B}, \textbf{23} (1977), 247-–250.

\bibitem{BK14}
O.V. Borodin and A.V. Kostochka. Defective 2-colorings of sparse graphs. \textit{J. Combin. Theory Ser. B}, \textbf{104} (2014), 72--80.

\bibitem{BKY13}
O.V. Borodin, A. Kostochka and M. Yancey, On 1-improper 2-coloring of sparse graphs. \textit{Discrete Math.} \textbf{313} (2013), no. 22, 2638--2649.

\bibitem{BL16}
A. Ban and N. Linial. Internal partitions of regular graphs. \textit{J. Graph Theory}. \textbf{83} (2016), 5--18.

\bibitem{BLM17}
R. Belmonte, M. Lampis and V. Mitsou. Defective coloring on classes of perfect graphs. \textit{Graph-theoretic concepts in computer science}, 113--126, Lecture Notes in Comput. Sci., 10520, Springer, Cham, 2017.

\bibitem{BLM20}
R. Belmonte, M. Lampis and V. Mitsou. Parameterized (approximate) defective coloring. \textit{SIAM J. Discrete Math.} \textbf{34} (2020), no. 2, 1084--1106.

\bibitem{BS02}
 B. Bollobás and A.D. Scott. Problems and results on judicious partitions. \textit{Random Structures Algorithms}, \textbf{21} (2002), 414–-430.

\bibitem{BTV07}
C. Bazgan, Z. Tuza and D. Vanderpooten. Efficient algorithms for decomposing graphs under degree constraints.
\textit{Discrete Appl. Math.}, \textbf{155} (2007), 979--988.

\bibitem{CCW86}
L.J. Cowen, R.H. Cowen and D.R. Woodall. Defective colorings of graphs in surfaces: partitions into subgraphs of bounded valency. \textit{J. Graph Theory}, \textbf{10} (1986), 187--195.

\bibitem{CE19}
I. Choi and L. Esperet. Improper coloring of graphs on surfaces. \textit{J. Graph Theory}, \textbf{91} (2019), no. 1, 16--34.

\bibitem{CGJ97}
L. Cowen, W. Goddard and C.E. Jesurum. Defective coloring revisited. \textit{J. Graph Theory}, \textbf{24} (1997), no. 3, 205--219.

\bibitem{CLNWY15}
E. Cs\'{o}ka, I. Lo, S. Norin, H. Wu and L. Yepremyan. The extremal function for disconnected minors. \emph{J. Combin. Theory Ser. B}, \textbf{126} (2017), 162--174. 

\bibitem{CLO18}
I. Choi, C.-H. Liu and S. Oum. Characterization of cycle obstruction sets for improper coloring planar graphs. \textit{SIAM J. Discrete Math.}, \textbf{32}  (2018), no. 2, 1209--1228. 

\bibitem{CSY19}
Y. Chu, L. Sun and J. Yue. Note on improper coloring of 1-planar graphs. \textit{Czechoslovak Math. J.}, \textbf{69} (144) (2019), no. 4, 955--968.

\bibitem{Di00}
A.A. Diwan. Decomposing graphs with girth at least five under degree constraints.
\textit{J. Graph Theory}, \textbf{33} (2000), 237--239.

\bibitem{FH94}
M. Frick and M.A. Henning. Extremal results on defective colorings of graphs. \textit{Discrete Math.}, \textbf{126} (1994), no. 1--3, 151--158.

\bibitem{Ge65}
L. Gerencsér. Szinezesi problemacrol. \textit{Mat. Lapok}, \textbf{16} (1965), 274--277.

\bibitem{GK00}
M.U. Gerber and D. Kobler. Algorithmic approach to the satisfactory graph partitioning problem. \textit{European J. Oper. Res.}, \textbf{125} (2000), 283--291.

\bibitem{GK04}
M.U. Gerber and D. Kobler. Classes of graphs that can be partitioned to satisfy all their vertices. \textit{Australas. J. Comb.}, \textbf{29} (2004), 201--214.

\bibitem{HKS09}
F. Havet, R. Kang and J.-S. Sereni. Improper coloring of unit disk graphs. \textit{Networks} \textbf{54} (2009), no. 3, 150--164.

\bibitem{HMYZ18}
J. Hou, H. Ma, J. Yu and X. Zhang. On partitions of $K_{2,3}$-free graphs under degree constraints. \textit{Discrete Math.}, \textbf{341} (2018), 3288--3295.

\bibitem{Ka98}
A. Kaneko. On decomposition of triangle-free graphs under degree constraints. \textit{J. Graph Theory}, \textbf{27} (1998), 7--9.

\bibitem{KKX14}
J. Kim, A. Kostochka and X. Zhu. Improper coloring of sparse graphs with a given girth, I: (0,1)-colorings of triangle-free graphs. \textit{European J. Combin.} \textbf{42} (2014), 26--48. 

\bibitem{KKX16}
J. Kim, A. Kostochka and X. Zhu. Improper coloring of sparse graphs with a given girth, II: constructions. \textit{J. Graph Theory} \textbf{81} (2016), no. 4, 403--413.

\bibitem{LL20}
 N. Linial and S. Louis. Asymptotically almost every 2r-regular graph has an internal partition. \textit{Graphs Combin.}, \textbf{36} (2020), 41--50.

\bibitem{LX17}
M. Liu and B. Xu. On partitions of graphs under degree constraints.
\textit{Discrete Appl. Math.}, \textbf{226} (2017), 87--93.

\bibitem{LX21}
M. Liu and B. Xu. On a conjecture of Schweser and Stiebitz.
\textit{Discrete Appl. Math.}, \textbf{295} (2021), 25--31.

\bibitem{LX22}
M. Liu, and B. Xu. On connected partition with degree constraints. \textit{Discrete Mathematics} \textbf{345} (2022), 112680.

\bibitem{Lo66}
L. Lov\'{a}sz. On decomposition of graphs. \emph{Studia Sci. Math. Hungal}, \textbf{1} (1966), 237--238. 

\bibitem{MY19}
J. Ma and T. Yang. Decomposing $C_4$-free graphs under degree constraints.
\textit{J. Graph Theory}, \textbf{90} (2019), 13--23.

\bibitem{Ra11}
T. Rackham. The number of defective colorings of graphs on surfaces. \textit{J. Graph Theory}, \textbf{68} (2011), no. 2, 129--136.

\bibitem{SD02}
K. H. Shafique and R. D. Dutton. On satisfactory partitioning of graphs. \textit{Proceedings of the Thirty-third Southeastern International Conference on Combinatorics, Graph Theory and Computing (Boca Raton, FL, 2002)}, \textbf{154} (2002), 183--194.

\bibitem{Sh12}
Y. Shang. Improper coloring of random geometric graphs. \textit{J. Adv. Res. Appl. Math.} \textbf{4} (2012), no. 1, 1--9.

\bibitem{SN18}
P. Sittitrai and K. Nakprasit. Defective 2-colorings of planar graphs without 4-cycles and 5-cycles. \textit{Discrete Math.} \textbf{341} (2018), no. 8, 2142--2150.

\bibitem{SS19}
T. Schweser and M. Stiebitz. Partitions of multigraphs under minimum degree constraints.
\textit{Discrete Appl. Math.}, \textbf{257} (2019), 269--275.

\bibitem{St96}
M. Stiebitz. Decomposing graphs under degree constraints. \emph{J.
Graph Theory}, \textbf{23} (3): 321--324, 1996.

\bibitem{SX20}
J. Song and B. Xu. Partitions of graphs and multigraphs under degree constraints.
\textit{Discrete Appl. Math.}, \textbf{279} (2020), 134--145.

\bibitem{SX21}
J. Song, and B. Xu. Partitions of Multigraphs Without $C_4$. \textit{Graphs and Combinatorics}, (2021), 1--17.

\bibitem{Th83}
C. Thomassen. Graph decomposition with constraints on the connectivity and minimum degree. \textit{J. Graph Theory}, \textbf{7} (1983), 165--167.

\end{thebibliography}
\end{document}